\documentclass[12pt]{amsart}
\usepackage{hyperref}
\usepackage[all]{xy}
\usepackage{amsmath,amssymb,amsthm}
\usepackage{mathrsfs}
\usepackage{txfonts}
\usepackage{enumerate}
\usepackage{graphicx}
\theoremstyle{plain}
	\newtheorem{thm}{Theorem}[section]
	\newtheorem{prp}[thm]{Proposition}
	\newtheorem{lem}[thm]{Lemma}

\theoremstyle{definition}
	\newtheorem{dfn}[thm]{Definition}
	\newtheorem{ex}[thm]{Example}
	\newtheorem{question}[thm]{Question}
\theoremstyle{remark}
	\newtheorem{rem}[thm]{Remark}

 \usepackage{amscd}

\newcommand{ \B}{\mathcal{B}}

\newcommand{ \U}{\operatorname{U}}

\newcommand{ \GL}{\operatorname{GL}}

\newcommand{\ind}{\operatorname{ind}}
\newcommand{\prop}{\operatorname{prop}}

\newcommand{\tsr}{\operatorname{tsr}}

\title{Homotopy type of the unitary group of the uniform Roe algebra on $\mathbb{Z}^n$}

\author{Tsuyoshi Kato}
\address{Department of Mathematics, Kyoto University, Kyoto, 606-8502, Japan}
\email{tkato@math.kyoto-u.ac.jp}

\author{Daisuke Kishimoto}
\address{Department of Mathematics, Kyoto University, Kyoto, 606-8502, Japan}
\email{kishi@math.kyoto-u.ac.jp}

\author{Mitsunobu Tsutaya}
\address{Faculty of Mathematics, Kyushu University, Fukuoka, 819-0395, Japan}
\email{tsutaya@math.kyushu-u.ac.jp}



\thanks{Kato was supported by JSPS KAKENHI 17K18725
and 17H06461. Kishimoto was supported by JSPS KAKENHI 17K05248 and 19K03473. Tsutaya was supported by JSPS KAKENHI 19K14535}

\begin{document}
\begin{abstract}
We study the homotopy type of the space of the unitary group $\U_1(C^\ast_u(|\mathbb{Z}^n|))$ of the uniform Roe algebra $C^\ast_u(|\mathbb{Z}^n|)$ of $\mathbb{Z}^n$.
We show that the stabilizing map $\U_1(C^\ast_u(|\mathbb{Z}^n|))\to\U_\infty(C^\ast_u(|\mathbb{Z}^n|))$ is a homotopy equivalence.
Moreover, when $n=1,2$, we determine the homotopy type of $\U_1(C^\ast_u(|\mathbb{Z}^n|))$, which is the product of the unitary group $\U_1(C^\ast(|\mathbb{Z}^n|))$ (having the homotopy type of $\U_\infty(\mathbb{C})$ or $\mathbb{Z}\times B\U_\infty(\mathbb{C})$ depending on the parity of $n$) of the Roe algebra $C^\ast(|\mathbb{Z}^n|)$ and rational Eilenberg--MacLane spaces.

\end{abstract}
\maketitle
%\tableofcontents
\section{Introduction}
For a $C^\ast$-algebra $A$, let $\GL_d(A)$ and $\U_d(A)$ denote the space of the invertible and unitary matrices with entries in $A$, respectively.
It is well-known that they always have the same homotopy type.
We will often refer only to $\U_d(A)$ but most statements are valid for $\GL_d(A)$ as well.
There have been a lot of works on the homotopy theory of $\U_d(A)$ and some of them have important applications.
For finite-dimensional case, the complex-valued unitary matrices $\U_d(\mathbb{C})$ is just the usual unitary group acting linearly on $\mathbb{C}^d$.
For inifinite-dimensional case, Kuiper \cite{MR179792} proved that the space of all unitary operators on an inifinite-dimensional Hilbert space is contractible.
This result is basic in the Atiyah--Singer index theory.
This kind of contractibility result has been extended to $\U_d(A)$ of some other algebras $A$ while $A$ is all the bounded operators on a infinite dimensional Hilbert space in the original result. 
Of course it is not always the case for $\U_d(A)$ of other infinite-dimensional $C^\ast$-algebras $A$.
In general, it is hard to determine the homotopy type of $\U_d(A)$.

Let us use the notation
\[
\GL_\infty(A)=\lim_{d\to\infty}\GL_d(A)
\quad\text{and}\quad
\U_\infty(A)=\lim_{d\to\infty}\U_d(A).
\]
It is well-known that $\U_\infty(A)$ has the same homotopy type as $\U_1(A\otimes\mathcal{K})$ where $\mathcal{K}$ is the space of compact operators.
The $K$-theory $K_i(A)$ ($i=0,1$) is a basic homotopy invariant of $A$, which is characterized as
\[
K_0(A)=\pi_1(\U_\infty(A))
\quad\text{and}\quad
K_1(A)=\pi_0(\U_\infty(A)).
\]
Since $\U_d(A)$ is not necessarily homotopy equivalent to $\U_\infty(A)$, $K_i(A)$ is not a so strong invariant in general.
But sometimes the natural map $\U_d(A)\to\U_\infty(A)$, which we will call the \textit{stabilizing map}, becomes a homotopy equivalence.
Study on such stability can be traced back to the work of Bass \cite{MR174604}.
There have been a various works on this kind of stability.
Rieffel introduced the \textit{topological stable rank} in \cite{MR693043} and applied it to show the stability of the non-commutative torus in \cite{MR887221}, which is a key tool in the present work.
It is difficult in general to determine how stable a given $C^\ast$-algebra is.

In the present paper, we study the stability of the uniform Roe algebra $C^\ast_u(|\mathbb{Z}^n|)$ on $\mathbb{Z}^n$ and investigate its homotopy type.
The \textit{uniform Roe algebra} $C^\ast_u(X)$ of a metric space $X$ is introduced by Roe in \cite{MR918459} to establish an index theory on open manifolds, where the index lives in the $K$-theory $K_\ast(C^\ast_u(X))$.
The algebra $C^\ast_u(X)$ itself is also important since it encodes a kind of ``large scale geometry'' of $X$.
Studying the homotopy type of $\U_d(C^\ast_u(X))$ will provide more insights from a homotopy theoretic viewpoint, which cannot be obtained only from its $K$-theory.
But there are only a few works on the homotopy type of $\U_d(C^\ast_u(X))$ yet.
For example, Manuilov and Troitsky \cite{MR4155292} studied some condition for $\U_d(C^\ast_u(X))$ being contractible.
In the present work, we observe the other extreme, that is, $\U_d(C^\ast_u(|\mathbb{Z}^n|))$ has a highly nontrivial homotopy type.

We give some comment on the relation with our previous work \cite{KKT} on the space $\mathcal{U}$ of finite propagation unitary operators on $\mathbb{Z}$.
Note that $\U_1(C^\ast_u(|\mathbb{Z}|))$ can be viewed as a kind of completion of $\mathcal{U}$.
We determined the homotopy type of $\mathcal{U}$ there.
But it is not clear whether $\mathcal{U}$ has the same homotopy type as $\U_1(C^\ast_u(|\mathbb{Z}|))$.
Actually, they turn out to have the same homotopy type (Theorem \ref{thm_decomp}).
Also, the method there does not seem to be extended to $\mathbb{Z}^n$ when $n\ge2$.
We employ rather operator algebraic technique in the present paper.
Our method here reduces the problem to determine the homotopy type of $\U_d(C^\ast_u(|\mathbb{Z}^n|))$ to the one to show the surjectivity of the homomorphism on $K$-theory $K_\ast(C^\ast_u(|\mathbb{Z}^n|))\to K_\ast(C^\ast(|\mathbb{Z}^n|))$ induced from the inclusion (Proposition \ref{prp_section_surj}), where $C^\ast(|\mathbb{Z}^n|)$ denotes the Roe algebra of $\mathbb{Z}^n$.

For stability, we show the following theorem in Section \ref{section_stability}.
\begin{thm}
\label{thm_stability}
For any integer $n\ge1$, the stabilizing maps
\begin{align*}
&\GL_1(C^\ast_u(|\mathbb{Z}^n|))\to\GL_\infty(C^\ast_u(|\mathbb{Z}^n|)),
&&\U_1(C^\ast_u(|\mathbb{Z}^n|))\to\U_\infty(C^\ast_u(|\mathbb{Z}^n|)),\\
&\GL_1(C^\ast(|\mathbb{Z}^n|))\to\GL_\infty(C^\ast(|\mathbb{Z}^n|)),
&&\U_1(C^\ast(|\mathbb{Z}^n|))\to\U_\infty(C^\ast(|\mathbb{Z}^n|))
\end{align*}
between the spaces of invertible and unitary elements are homotopy equivalences.
This implies that these maps induce the following isomorphisms on homotopy groups for all $i\ge0$:
\begin{align*}
&\pi_i(\GL_1(C^\ast_u(|\mathbb{Z}^n|)))
\cong
\pi_i(\U_1(C^\ast_u(|\mathbb{Z}^n|)))
\cong
\begin{cases}
K_1(C^\ast_u(|\mathbb{Z}^n|)) & \text{$i$ is even,} \\
K_0(C^\ast_u(|\mathbb{Z}^n|)) & \text{$i$ is odd,}
\end{cases}
\\
&\pi_i(\GL_1(C^\ast(|\mathbb{Z}^n|)))
\cong
\pi_i(\U_1(C^\ast(|\mathbb{Z}^n|)))
\cong
\begin{cases}
K_1(C^\ast(|\mathbb{Z}^n|)) & \text{$i$ is even,} \\
K_0(C^\ast(|\mathbb{Z}^n|)) & \text{$i$ is odd.}
\end{cases}
\end{align*}
\end{thm}

Let $K(V,i)$ denote the Eilenberg--MacLane space of type $(V,i)$ and $B\U_\infty(\mathbb{C})$ denote the classifying space of the unitary group $\U_\infty(\mathbb{C})$.
Also, for based spaces $X_i$ ($i=1,2,\ldots$), define
\[
\prod_{i\ge1}^\circ X_i=\lim_{k\to\infty}(X_1\times X_2\times\cdots\times X_k).
\]
For the homotopy type of $\U_1(C^\ast_u(|\mathbb{Z}^n|))$, we show the following results when $n=1,2$.
\begin{thm}
\label{thm_decomp}
There exist homotopy equivalences of infinite loop spaces
\[
\GL_1(C^\ast_u(|\mathbb{Z}|))
\simeq
\U_1(C^\ast_u(|\mathbb{Z}|))
\simeq
\mathbb{Z}\times B\U_\infty(\mathbb{C})
	\times\prod_{i\ge1}^\circ
	K(\ell^\infty(\mathbb{Z},\mathbb{Z})_{S},2i-1).
\]
\end{thm}

\begin{thm}
\label{thm_decomp_Z2}
There exist homotopy equivalences of infinite loop spaces
\[
\GL_1(C^\ast_u(|\mathbb{Z}^2|))
\simeq
\U_1(C^\ast_u(|\mathbb{Z}^2|))
\simeq
V_1\times\U_\infty(\mathbb{C})\times
\prod_{i\ge1}^\circ(K(V_0,2i-1)\times K(V_1,2i)),
\]
where $V_0,V_1$ are the rational vector spaces given by
\[
V_0
=
\ker[K_0(C^\ast_u(|\mathbb{Z}^2|))\to K_0(C^\ast(|\mathbb{Z}^2|))],
\quad
V_1
=
K_1(C^\ast_u(|\mathbb{Z}^2|))
\]
and the product factor $V_1$ is a discrete space.
\end{thm}

More detailed descriptions of the vector spaces $V_0$ and $V_1$ appear in the proof of Lemma \ref{lem_KCuZ2}.

We will see the existence of a homotopy section of the inclusion $\U_1(C^\ast_u(|\mathbb{Z}|))\to\U_1(B^{\mathrm{SW}})$ in Section \ref{section_htpytype}, where $\U_1(B^{\mathrm{SW}})$ is the Segal--Wilson restricted unitary group \cite{MR783348} having the homotopy type of $\mathbb{Z}\times B\U_\infty(\mathbb{C})$.
This implies Theorem \ref{thm_decomp}.
We also show in Section \ref{section_generalization} that, for any integer $n\ge1$, the inclusion $\U_1(C^\ast_u(|\mathbb{Z}^n|))\to\U_1(C^\ast(|\mathbb{Z}^n|))$ admits a homotopy section if and only if the homomorphism $K_\ast(C^\ast_u(|\mathbb{Z}^n|))\to K_\ast(C^\ast(|\mathbb{Z}^n|))$ is surjective (Proposition \ref{prp_section_surj}).
Since we can see it is surjective when $n=1,2$, Theorems \ref{thm_decomp} again and \ref{thm_decomp_Z2} follows.
If one could show the surjectivity for $n\ge3$, then a similar homotopy decomposition will immediately follow.

This paper is organized as follows.
We fix our notation in Section \ref{section_finprop}.
In Section \ref{section_stability}, we recall Rieffel's results on stability and show Theorem \ref{thm_stability}.
In Section \ref{section_Bott}, we recall the Bott periodicity realized as a $\ast$-homomorphism.
In Section \ref{section_SW}, we recall the Segal--Wilson restricted unitary group and show its stability.
In Section \ref{section_htpytype}, we show Theorem \ref{thm_decomp} using the Segal--Wilson restricted unitary group.
In Section \ref{section_generalization}, we discuss the homotopy type of $\U_d(C^\ast_u(|\mathbb{Z}^n|))$ for general $n\ge1$ and show Theorems \ref{thm_decomp} again and \ref{thm_decomp_Z2}.

\section{Notation}
\label{section_finprop}

The $C^\ast$-algebra of bounded operators on a Hilbert space $V$ is denoted by $\mathcal{B}(V)$ and the subalgebra of compact operators by $\mathcal{K}(V)$.
We write the operator norm of $T\in\mathcal{B}(V)$ as $\|T\|$.
The Hilbert space of square summable sequences indexed by a discrete group $\Gamma$ will be written as
\[
\ell^2(\Gamma)
=
\{(v_g)_g\mid
\sum_{g\in\Gamma}|v_g|^2<\infty\}.
\]
We also consider the tensor product Hilbert space $\ell^2(\Gamma)\otimes\mathcal{H}$ with an infinite dimensional separable Hilbert space $\mathcal{H}$.

A bounded operator $T\in\mathcal{B}(\ell^2(\Gamma))$ can be expressed in the matrix form as
\[
T=(T_{g,h})_{g,h},
\quad
T_{g,h}\in\mathbb{C}.
\]
For $T\in\mathcal{B}(\ell^2(\Gamma)\otimes\mathcal{H})$, we also have a similar expression $T=(T_{g,h})_{g,h}$ with $T_{g,h}\in\mathcal{B}(\mathcal{H})$.

\begin{dfn}
Let $\Gamma$ be a finitely generated group and $d$ denote the word metric with respect to some finite set of generators.
We say that a bounded operator $T\in\B(\ell^2(\Gamma))$ has \textit{finite propagation} if
\[
\prop(T)= \sup\{d(g,h)\mid T_{g,h}\ne0\}
\]
is finite.
We define finite propagation for $T=(T_{g,h})_{g,h}\in\mathcal{B}(\ell^2(\Gamma)\otimes\mathcal{H})$ similarly.
\end{dfn}

\begin{ex}
\label{ex_shift}
The \textit{shift} $S_x\in\mathcal{B}(\ell^2(\Gamma))$ by $x\in\Gamma$ is defined by
\[
S_x=((S_x)_{g,h})_{g,h},
\quad
S_{g,h}=
\begin{cases}
1 & g^{-1}h=x, \\
0 & \text{otherwise.}
\end{cases}
\]
The operator $S_x$ is a unitary operator with $\prop(S_x)=d(x,1)$.
\end{ex}

It is easy to see that the definition of having finite propagation is independent of the choice of generators while the value of $\prop(T)$ depends on the word metric.
Since we have
\[
\prop(ST)\le\prop(S)+\prop(T),
\quad
\prop(T^\ast)=\prop(T),
\quad
\prop(1)=0
\]
for any finite propagation operators $S,T\in\mathcal{B}(\ell^2(\Gamma))$, the subset of finite propagation operators becomes a unital $\ast$-subalgebra of $\mathcal{B}(\ell^2(\Gamma))$.
Similar properties hold for finite propagation operators $S,T\in\mathcal{B}(\ell^2(\Gamma)\otimes\mathcal{H})$ such that the components $T_{g,h}$ and $S_{g,h}$ are compact operators.

\begin{dfn}
The \textit{uniform Roe algebra} $C^\ast_u(|\Gamma|)$ of $\Gamma$ is the norm closure of the algebra of finite propagation operators in $\mathcal{B}(\ell^2(\Gamma))$.
\end{dfn}

\begin{dfn}
The \textit{Roe algebra} $C^\ast(|\Gamma|)$ of $\Gamma$ is the norm closure of the algebra of finite propagation operators $T\in\mathcal{B}(\ell^2(\Gamma)\otimes\mathcal{H})$ such that each component $T_{g,h}$ is a compact operator.
\end{dfn}

\begin{rem}
We follow the usual notation $C^\ast(|\Gamma|)$ for the Roe algebra of $\Gamma$ to distinguish it from the group $C^\ast$-algebra of $\Gamma$ though we do not consider the latter here.
\end{rem}

We will consider the uniform Roe algebra $C^\ast_u(|\Gamma|)$ is a subalgebra of the Roe algebra $C^\ast(|\Gamma|)$ with respect to some inclusion $\mathbb{C}\subset\mathcal{H}$.

We use the symbol $\ell^\infty(\Gamma,\mathbb{C})$ to express the Banach algebra of $\mathbb{C}$-valued bounded sequences indexed by $\Gamma$ rather than the simpler symbol $\ell^\infty(\Gamma)$ since we also consider the abelian group of $\mathbb{Z}$-valued bounded sequences $\ell^\infty(\Gamma,\mathbb{Z})$.

The group $\Gamma$ acts on the algebras $\ell^\infty(\Gamma,\mathbb{C})$ and $\ell^\infty(\Gamma,\mathcal{K}(\mathcal{H}))$ by right translation.
The action by $x\in\Gamma$ is compatible with the conjugation by $S_x$ through the diagonal inclusion $\ell^\infty(\Gamma,\mathbb{C})\to C^\ast_u(|\Gamma|)$ or $\ell^\infty(\Gamma,\mathcal{K}(\mathcal{H}))\to C^\ast(|\Gamma|)$ given by
\[
(t_g)_g\mapsto(T_{g,h})_{g,h},
\quad
T_{g,h}=
\begin{cases}
t_g & g=h,\\
0 & \text{otherwise.}
\end{cases}
\]
Moreover, this inclusion extends to the well-known isomorphisms
\[
\ell^\infty(\Gamma,\mathbb{C})\rtimes\Gamma\cong C^\ast_u(|\Gamma|),
\quad
\ell^\infty(\Gamma,\mathcal{K}(\mathcal{H}))\rtimes\Gamma\cong C^\ast(|\Gamma|)
\]
from the reduced crossed products of $C^\ast$-algebras.
For example, see \cite[Theorem 4.28]{MR2007488}.

The $d\times d$-matrix algebra $M_d(A)$ of a $C^\ast$-algebra $A$ is again a $C^\ast$-algebra.
The spaces of invertible elements and unitary elements in $M_d(A)$ will be denoted as $\GL_d(A)$ and $\U_d(A)$.
The \textit{stabilizing maps} are given as
\[
\GL_1(A)\to\GL_\infty(A)=\lim_{d\to\infty}\GL_d(A),
\quad
\U_1(A)\to\U_\infty(A)=\lim_{d\to\infty}\U_d(A),
\]
where the inductive limits are taken along the inclusions $\GL_d(A)\subset\GL_{d+1}(A)$ and $\U_d(A)\subset\U_{d+1}(A)$.
The inductive limit spaces $\GL_\infty(A)$ and $\U_\infty(A)$ are well-known to be homotopy equivalent to the spaces $\GL_1(A\otimes\mathcal{K}(\mathcal{H}))$ and $\U_1(A\otimes\mathcal{K}(\mathcal{H}))$.

\section{Stability}
\label{section_stability}

The aim of this section is to prove Theorem \ref{thm_stability}.
Once the assumption of the following result by Rieffel \cite{MR887221} is verified, the theorem will immediately follow.
\begin{thm}[Rieffel]
\label{thm_rieffel_stability}
Let $A$ be a unital $C^\ast$-algebra.
If $A$ is tsr-boundedly divisible, then the stabilizing maps
\[
\GL_1(A)\to\GL_\infty(A)
\quad
\text{and}
\quad
\U_1(A)\to\U_\infty(A)
\]
are homotopy equivalences.
\end{thm}
\begin{rem}
The original statement of Theorem 4.13 in \cite{MR887221} is involved only with homotopy groups.
But what is actually proved there is slightly stronger as above.
\end{rem}

For the definitions of the \textit{topological stable rank} $\tsr(A)\in\mathbb{Z}_{\ge1}$, see \cite{MR693043}.
A $C^\ast$-algebra $A$ is said to be \textit{tsr-boundedly divisible} \cite{MR887221} if there is a constant $K$ such that for any integer $m$, there exists an integer $d\ge m$ such that $A$ is isomorphic to $M_d(B)$ for some $C^\ast$-algebra $B$ with $\tsr(B)\le K$.
To verify the assumption, we need the following two lemmas.
\begin{lem}
\label{lem_tsr_Roe}
The topological stable ranks of $C^\ast_u(|\mathbb{Z}^n|)$ and $C^\ast(|\mathbb{Z}^n|)$ are estimated as
\[
\tsr(C^\ast_u(|\mathbb{Z}^n|))\le n+1
\quad
\text{and}
\quad
\tsr(C^\ast(|\mathbb{Z}^n|))\le n+1.
\]
\end{lem}
\begin{rem}
We will see that both $C^\ast_u(|\mathbb{Z}^n|)$ and $C^\ast(|\mathbb{Z}^n|)$ are tsr-boundedly divisible using this lemma.
Thus we will actually obtain the estimates $\tsr(C^\ast_u(|\mathbb{Z}^n|))\le2$ and $\tsr(C^\ast(|\mathbb{Z}^n|))\le2$ by \cite[Proposition 4.6]{MR887221}.
\end{rem}
\begin{proof}
Let $A=\mathbb{C}$ or $\mathcal{K}(\mathcal{H})$.
Since the invertible elements in $\ell^\infty(\mathbb{Z}^n,\mathbb{C})$ and $\mathbb{C}\oplus\ell^\infty(\mathbb{Z}^n,\mathcal{K}(\mathcal{H}))$ are dense, we have
\[
\tsr(\ell^\infty(\mathbb{Z}^n,A))=1
\]
by \cite[Proposition 3.1]{MR693043}.
Considering the restricted action of $\mathbb{Z}^m\subset\mathbb{Z}^n$ on the first $m$ factors of $\mathbb{Z}^n$, we obtain the isomorphism
\[
\ell^\infty(\mathbb{Z}^n,A)\rtimes\mathbb{Z}^{m+1}
\cong
(\ell^\infty(\mathbb{Z}^n,A)\rtimes\mathbb{Z}^m)\rtimes\mathbb{Z}.
\]
Thus, by \cite[Theorem 7.1]{MR693043}, we get the desired estimates on $\tsr(C^\ast_u(|\mathbb{Z}^n|))$ and $\tsr(C^\ast(|\mathbb{Z}^n|))$.
\end{proof}
\begin{lem}
\label{lem_divisibilityRoe}
For any integer $d\ge1$, there exist isomorphisms
\[
\phi\colon C^\ast_u(|\mathbb{Z}^n|)\cong M_d(C^\ast_u(|\mathbb{Z}^n|))
\quad
\text{and}
\quad
\phi\colon C^\ast(|\mathbb{Z}^n|)\cong M_d(C^\ast(|\mathbb{Z}^n|)).
\]
\end{lem}
\begin{proof}
Let $V=\mathbb{C}$ or $\mathcal{H}$.
According to the decomposition
\[
\ell^2(\mathbb{Z}^n)\otimes V
=
\bigoplus_{(i_1,\ldots,i_n)\in\mathbb{Z}^n}V_{(i_1,\ldots,i_n)},
\quad
V_{(i_1,\ldots,i_n)}
\cong V,
\]
we have the matrix expression for $T\in\mathcal{B}(\ell^2(\mathbb{Z}^n)\otimes V)$
\[
T=(T_{(i_1,\ldots,i_n)(j_1,\ldots,j_n)})_{(i_1,\ldots,i_n)(j_1,\ldots,j_n)},
\quad
T_{(i_1,\ldots,i_n)(j_1,\ldots,j_n)}
\colon
V_{(j_1,\ldots,j_n)}\to V_{(i_1,\ldots,i_n)}.
\]
Consider the map $\phi\colon\mathcal{B}(\ell^2(\mathbb{Z}^n)\otimes V)\to M_d(\mathcal{B}(\ell^2(\mathbb{Z}^n)\otimes V))$ given by
\[
\phi(T)_{(i_1,\ldots,i_n)(j_1,\ldots,j_n)}=
\begin{pmatrix}
T_{(di_1,\ldots,i_n)(dj_1,\ldots,j_n)} & \cdots & T_{(di_1,\ldots,i_n)(dj_1+d-1,\ldots,j_n)} \\
\vdots & \ddots & \vdots \\
T_{(di_1+d-1,\ldots,i_n)(dj_1,\ldots,j_n)} &\cdots & T_{(di_1+d-1,\ldots,i_n)(dj_1+d-1,\ldots,j_n)}
\end{pmatrix}
\in M_d(\mathcal{B}(V)).
\]
The restrictions to $C^\ast_u(|\mathbb{Z}^n|)\subset\mathcal{B}(\ell^2(\mathbb{Z}^n))$ and $C^\ast(|\mathbb{Z}^n|)\subset\mathcal{B}(\ell^2(\mathbb{Z}^n)\otimes\mathcal{H})$ are desired isomorphisms.
\end{proof}

\begin{rem}
When $n=1$, the map $\phi$ is just taking the block matrix of which each block is a $d\times d$-matrix.
\end{rem}

\begin{proof}[Proof of Theorem \ref{thm_stability}]
By Lemmas \ref{lem_tsr_Roe} and \ref{lem_divisibilityRoe}, we can apply Theorem \ref{thm_rieffel_stability} to $C^\ast_u(|\mathbb{Z}^n|)$ and $C^\ast(|\mathbb{Z}^n|)$.
This completes the proof of the theorem.
\end{proof}

\section{Bott periodicity}
\label{section_Bott}

Let us recall the Bott periodicity of $C^\ast$-algebras here.
Let $A$ be a $C^\ast$-algebra, which might be non-unital.
The direct sum $\mathbb{C}\oplus A$ is considered as the unitization with unit $(1,0)\in\mathbb{C}\oplus A$.
Define the unitary group $\U'_d(A)$ by
\[
\U'_d(A)=\{U\in\U_n(\mathbb{C}\oplus A)\mid U-(I_d,0)\in M_d(A)\}.
\]
If $A$ is already unital, we have a canonical isomorphism $\U'_d(A)\cong\U_d(A)$.
So we use the same symbol $\U_d(A)$ for $\U'_d(A)$ even if $A$ is not unital.

Consider the following space of continuous functions:
\[
C_0(\mathbb{R}^m,A)
=
\{T\colon\mathbb{R}^m\to A\mid T\text{ is continuous and }\lim_{|z|\to\infty}T(z)=0\}.
\]
This is a $C^\ast$-algebra without unit.
Notice that $C_0(\mathbb{R}^m,A)$ is isomorphic to the space $\Omega^mA$ of based maps from the $m$-sphere $S^m$ to $A$ where the basepoint $\ast\in S^m$ is mapped to $0\in A$.

Set the element
\[
p_B(z)=\frac{1}{1+|z|^2}
\begin{pmatrix}
|z|^2 & z \\
\bar{z} & 1
\end{pmatrix}
\in M_2(\mathbb{C}\oplus C_0(\mathbb{R}^2,\mathbb{C}))
\quad
(z\in\mathbb{R}^2),
\]
where we identify $\mathbb{R}^2\cong\mathbb{C}$ in the matrix entries.
The \textit{Bott map} $\beta\colon A\to M_2(A\oplus C_0(\mathbb{R}^2,A))$ is a $\ast$-homomorphism defined by
\[
\beta(a)=p_B
\begin{pmatrix}
a & 0 \\
0 & a
\end{pmatrix}
=
\begin{pmatrix}
a & 0 \\
0 & a
\end{pmatrix}
p_B.
\]
Then we have the commutative square of unital $C^\ast$-algebras
\[
\xymatrix{
\mathbb{C}\oplus A \ar[r]^-\epsilon \ar[d]_-\beta
& \mathbb{C} \ar[d]^-{\eta} \\
M_2(\mathbb{C}\oplus A\oplus C_0(\mathbb{R}^2,A)) \ar[r]^-\epsilon
& M_2(\mathbb{C}\oplus A)
}
\]
where $\epsilon\colon\mathbb{C}\oplus A\to\mathbb{C}$ and $\epsilon\colon M_2(\mathbb{C}\oplus A\oplus C_0(\mathbb{R}^2,A))\to M_2(\mathbb{C}\oplus A)$ are the projections and $\eta\colon\mathbb{C}\to M_2(\mathbb{C}\oplus A)$ is the unit map.
This square induces the $\ast$-homomorphism between the kernels of $\epsilon$:
\[
\beta\colon A\to M_2(C_0(\mathbb{R}^2,A)).
\]
We call this $\beta$ the \textit{Bott map} as well.
It is natural in the following sense:
if $f\colon A\to B$ is a $\ast$-homomorphism between $C^\ast$-algebras, then the following square commutes:
\[
\xymatrix{
A \ar[r]^-{f} \ar[d]_-\beta
& B \ar[d]^-\beta \\
M_2(C_0(\mathbb{R}^2,A)) \ar[r]^-{f_\ast}
& M_2(C_0(\mathbb{R}^2,B))
}
\]

\begin{prp}
\label{prp_Bottper}
The Bott map $\beta\colon A\to M_2(C_0(\mathbb{R}^2,A))$ induces an isomorphism on $K$-theory.
\end{prp}

\begin{rem}
This can be seen as a formulation of the Bott periodicity.
If you wish to deduce this proposition from the results appearing in \cite{MR859867}, it follows from the observation 9.2.10 on the generator of $K_0(C_0(\mathbb{R}^2,\mathbb{C}))$ and the K\"{u}nneth theorem for tensor products (Theorem 23.1.3).
\end{rem}

The Bott periodicity provides the natural homotopy equivalence
\[
\U_\infty(A)
\xrightarrow{\beta}
\U_\infty(M_2(C_0(\mathbb{R}^2,A)))
\simeq
\Omega^2\U_\infty(A),
\]
which is a group homomorphism.
Thus we obtain the following proposition on infinite loop structure.

\begin{prp}
\label{prp_infloop}
The unitary group $\U_\infty(A)$ of a $C^\ast$-algebra $A$ is equipped with a canonical infinite loop space structure such that the map $\U_\infty(A)\to\U_\infty(B)$ induced from a $\ast$-homomorphism $A\to B$ is an infinite loop map.
Moreover, the underlying loop structure of $\U_\infty(A)$ coincides with the group structure of $\U_\infty(A)$.
\end{prp}

\begin{rem}
The last sentence in the proposition means that the classifying space $B\U_\infty(A)$ of the topological group $\U_\infty(A)$ is homotopy equivalent to the identity component of $\Omega\U_\infty(A)$.
\end{rem}

\section{Segal--Wilson restricted unitary group}
\label{section_SW}
To study the homotopy type of $\U_1(C^\ast_u(|\mathbb{Z}|))$, we will relate it with other spaces.
One is the Segal--Wilson restricted unitary group $\U_1(B^{\mathrm{SW}})$ and the other is the unitary group of the Roe algebra $\U_1(C^\ast(|\mathbb{Z}|))$.
We recall the former in this section.

We have another matrix expression for $T\in \mathcal{B}(\ell^2(\mathbb{Z}))$ as
\[
T=
\begin{pmatrix}
T_{--} & T_{-+} \\
T_{+-} & T_{++}
\end{pmatrix},
\]
where
\begin{align*}
& T_{--}\colon\ell^2(\mathbb{Z}_{<0})\to\ell^2(\mathbb{Z}_{<0}),
\quad  T_{-+}\colon\ell^2(\mathbb{Z}_{\ge0})\to\ell^2(\mathbb{Z}_{<0}), \\
& T_{+-}\colon\ell^2(\mathbb{Z}_{<0})\to\ell^2(\mathbb{Z}_{\ge0}),
\quad T_{++}\colon\ell^2(\mathbb{Z}_{\ge0})\to\ell^2(\mathbb{Z}_{\ge0}).
\end{align*}

\begin{dfn}
We define the $C^\ast$-algebra $B^{\mathrm{SW}}$ by
\begin{align*}
B^{\mathrm{SW}}:=
\{
T\in \B(\ell^2(\mathbb{Z}))\mid\text{$T_{-+},T_{+-}$ are compact}
\}.
\end{align*}
\end{dfn}

The symbol ``SW'' stands for Segal--Wilson.
The unitary group $\U_1(B^{\mathrm{SW}})$ is called the \textit{restricted unitary group} in the work of Segal and Wilson \cite{MR783348}.
They used it as a model of the infinite Grassmannian.

\begin{lem}[Segal--Wilson]
\label{lem_pi0_BSW}
The space $\U_1(B^{\mathrm{SW}})$ has the homotopy type of $\mathbb{Z}\times B\U_\infty(\mathbb{C})$.
Moreover, the map
\[
\pi_0(\U_1(B^{\mathrm{SW}}))\to\mathbb{Z},
\quad
[U]\mapsto\ind(U_{++}),
\]
is bijective, where $\ind(U_{++})$ denotes the Fredholm index of the Fredholm operator $U_{++}$.
\end{lem}

Let $S=S_{+1}\in B^{\mathrm{SW}}$ the shift operator as in Example \ref{ex_shift}.
We have $\ind S^n=n$.

The goal of this section is to see the following.

\begin{prp}
\label{prp_stability_BSW}
The stabilizing maps
\[
\GL_1(B^{\mathrm{SW}})\to\GL_\infty(B^{\mathrm{SW}})
\quad\text{and}\quad
\U_1(B^{\mathrm{SW}})\to\U_\infty(B^{\mathrm{SW}})
\]
are homotopy equivalences.
\end{prp}

To show this, we do not use a kind of stability as in Section \ref{section_stability}. 

\begin{lem}
\label{lem_BSW_isom_pi0}
For any integer $d\ge1$, the inclusion
\[
\U_1(B^{\mathrm{SW}})\to\U_d(B^{\mathrm{SW}})
\]
induces an isomorphism on $\pi_0$.
\end{lem}

\begin{proof}
Consider the composite of the inclusion and the isomorphism $\phi\colon B^{\mathrm{SW}}\to M_d(B^{\mathrm{SW}})$ similar to the one in the proof of Lemma \ref{lem_divisibilityRoe}:
\[
\U_1(B^{\mathrm{SW}})\to\U_d(B^{\mathrm{SW}})\xrightarrow{\phi^{-1}}\U_1(B^{\mathrm{SW}}).
\]
It is easy to see that the image of the shift $S\in\U_1(B^{\mathrm{SW}})$ under this composite again has index $1$.
This implies the lemma.
\end{proof}

\begin{lem}
\label{lem_KBSW}
The $K$-theory of $B^{\mathrm{SW}}$ is computed as
\[
K_i(B^{\mathrm{SW}})
\cong
\begin{cases}
0 & i=0, \\
\mathbb{Z} & i=1,
\end{cases}
\]
where $K_1(B^{\mathrm{SW}})$ is generated by the shift $S\in\U_1(B^{\mathrm{SW}})$.
\end{lem}

\begin{proof}
This follows from the isomorphisms
\[
K_0(B^{\mathrm{SW}})
\cong
\lim_{d\to\infty}\pi_1(\U_d(B^{\mathrm{SW}}))
\quad\text{and}\quad
K_1(B^{\mathrm{SW}})
\cong
\lim_{d\to\infty}\pi_0(\U_d(B^{\mathrm{SW}}))
\]
and Lemmas \ref{lem_pi0_BSW} and \ref{lem_BSW_isom_pi0}.
\end{proof}

\begin{lem}
\label{lem_approx_Bott}
For any $i\ge0$, there exists an integer $m\ge1$ such that the iterated Bott map
\[
\beta^i\colon\U_d(B^{\mathrm{SW}})\to\U_d(M_{2^i}(C_0(\mathbb{R}^{2i},B^{\mathrm{SW}})))
\]
induces an isomorphism on $\pi_0$ if $d\ge m$.
\end{lem}

\begin{proof}
From the isomorphisms
\[
\pi_0(\U_d(M_{2^i}(C_0(\mathbb{R}^{2i},B^{\mathrm{SW}})))
\cong
\pi_{2i}(\U_{2^id}(B^{\mathrm{SW}}))
\cong
\pi_{2i}(\U_1(B^{\mathrm{SW}}))
\cong\mathbb{Z}
\]
and
\[
K_1(M_{2^i}(C_0(\mathbb{R}^{2i},B^{\mathrm{SW}})))
\cong
\lim_{d\to\infty}\pi_0(\U_d(M_{2^i}(C_0(\mathbb{R}^{2i},B^{\mathrm{SW}})))
\cong K_1(B^{\mathrm{SW}})\cong\mathbb{Z},
\]
we can find an integer $m\ge1$ such that the stabilizing map
\[
\pi_0(\U_d(M_{2^i}(C_0(\mathbb{R}^{2i},B^{\mathrm{SW}})))
\to
K_1(C_0(M_{2^i}(\mathbb{R}^{2i},B^{\mathrm{SW}})))
\]
is an isomorphism if $d\ge m$.
Consider the commutative diagram
\[
\xymatrix{
\pi_0(\U_d(B^{\mathrm{SW}})) \ar[r]^-{\cong} \ar[d]_-{(\beta^i)_\ast}
& K_1(B^{\mathrm{SW}}) \ar[d]^-{(\beta^i)_\ast} \\
\pi_0(\U_d(M_{2^i}(C_0(\mathbb{R}^{2i},B^{\mathrm{SW}}))) \ar[r]^-{\cong}
& K_1(M_{2^i}(C_0(\mathbb{R}^{2i},B^{\mathrm{SW}}))
}
\]
where the top arrow is an isomorphism by Lemma \ref{lem_BSW_isom_pi0} and the right Bott map $\beta^i$ is an isomorphism by Proposition \ref{prp_Bottper}.
Then the lemma follows.
\end{proof}

\begin{proof}[Proof of Proposition \ref{prp_stability_BSW}]
Take an integer $i\ge0$.
We can find an integer $m\ge1$ as in Proposition \ref{lem_approx_Bott} and
\[
\pi_{2i}(\U_d(B^{\mathrm{SW}}))\to\pi_{2i}(\U_\infty(B^{\mathrm{SW}}))\cong\mathbb{Z}
\]
is an isomorphism if $d\ge m$.
Consider the following commutative diagram:
\[
\xymatrix{
\U_1(C_0(\mathbb{R}^{2i},B^{\mathrm{SW}})) \ar[r]^-{\cong} \ar[d]
& \U_1(M_{2^i}(C_0(\mathbb{R}^{2i},B^{\mathrm{SW}}))) \ar[d]
& \U_1(B^{\mathrm{SW}}) \ar[l]_-{\beta^i} \ar[d]^{\text{isom. on $\pi_0$}} \\
\U_d(C_0(\mathbb{R}^{2i},B^{\mathrm{SW}})) \ar[r]^-{\cong}
& \U_d(M_{2^i}(C_0(\mathbb{R}^{2i},B^{\mathrm{SW}})))
& \U_d(B^{\mathrm{SW}}) \ar[l]_-{\beta^i}
}
\]
where the left horizontal arrows are the isomorphisms similar to the one in Lemma \ref{lem_divisibilityRoe} and the vertical arrows are the inclusions.
Since the composite
\[
\U_1(B^{\mathrm{SW}})
\to
\U_d(B^{\mathrm{SW}})
\xrightarrow{\beta^i}
\U_d(M_{2^i}(C_0(\mathbb{R}^{2i},B^{\mathrm{SW}})))
\]
induces an isomorphism on $\pi_0$, the middle vertical arrow
\[
\U_1(M_{2^i}(C_0(\mathbb{R}^{2i},B^{\mathrm{SW}})))
\to
\U_d(M_{2^i}(C_0(\mathbb{R}^{2i},B^{\mathrm{SW}})))
\]
induces a surjection on $\pi_0$.
But it is indeed an isomorphism as their $\pi_0$ are isomorphic to $\mathbb{Z}$.
Then the map
\[
\U_1(C_0(\mathbb{R}^{2i},B^{\mathrm{SW}}))
\to
\U_d(C_0(\mathbb{R}^{2n},B^{\mathrm{SW}}))
\]
induces an isomorphism on $\pi_0$.
This implies that the map
\[
\U_1(B^{\mathrm{SW}})\to\U_d(B^{\mathrm{SW}})
\]
induces an isomorphism on $\pi_{2i}$.
Thus the map
\[
\U_1(B^{\mathrm{SW}})\to\U_\infty(B^{\mathrm{SW}})
\]
induces an isomorphism on $\pi_{2i}$.
This completes the proof.
\end{proof}

\section{Homotopy type of \texorpdfstring{$\U_1(C^\ast_u(|\mathbb{Z}|))$}{U1(C*u(|Z|))}}
\label{section_htpytype}

The goal of this section is to prove Theorem \ref{thm_decomp}.
The components $T_{-+}$ and $T_{+-}$ of a finite propagation operator $T\in\mathcal{B}(H)$ are finite rank operators.
This implies the inclusion
\[
C^\ast_u(|\mathbb{Z}|)
\subset
B^{\mathrm{SW}}.
\]
This map is a key to the proof of Theorem \ref{thm_decomp}.

We begin with computing the $K$-theory.

\begin{prp}
\label{prp_K(CuZ)}
The following isomorphism holds:
\[
K_i(C^\ast_u(|\mathbb{Z}|))
\cong
\begin{cases}
\ell^\infty(\mathbb{Z},\mathbb{Z})_S & i=0,\\
\mathbb{Z} & i=1,
\end{cases}
\]
where
\[
\ell^\infty(\mathbb{Z},\mathbb{Z})_S
=
\ell^\infty(\mathbb{Z},\mathbb{Z})/\{a-Sa\mid a\in\ell^\infty(\mathbb{Z},\mathbb{Z})\}
\]
is the coinvariant by the shift $S\colon\ell^\infty(\mathbb{Z},\mathbb{Z})\to\ell^\infty(\mathbb{Z},\mathbb{Z})$.
\end{prp}

\begin{proof}
Applying the Pimsner--Voiculescu exact sequence \cite{MR587369} to the crossed product
\[
C^\ast_u(|\mathbb{Z}|)
\cong
\ell^\infty(\mathbb{Z},\mathbb{C})\rtimes\mathbb{Z},
\]
we get the six-term cyclic exact sequence:
\[
\xymatrix{
K_0(\ell^\infty(\mathbb{Z},\mathbb{C})) \ar[r]^-{1-S}
& K_0(\ell^\infty(\mathbb{Z},\mathbb{C})) \ar[r]
& K_0(C^\ast_u(|\mathbb{Z}|)) \ar[d] \\
K_1(C^\ast_u(|\mathbb{Z}|)) \ar[u]
& K_1(\ell^\infty(\mathbb{Z},\mathbb{C})) \ar[l]
& K_1(\ell^\infty(\mathbb{Z},\mathbb{C})) \ar[l]_-{1-S}
}
\]
As is well-known, we have
\[
K_i(\ell^\infty(\mathbb{Z},\mathbb{C}))
\cong
\begin{cases}
\ell^\infty(\mathbb{Z},\mathbb{Z}) & i=0,\\
0 & i=1,
\end{cases}
\]
where the induced homomorphism $S\colon\ell^\infty(\mathbb{Z},\mathbb{Z})\to\ell^\infty(\mathbb{Z},\mathbb{Z})$ is the shift as well.
Thus we can compute $K_\ast(C^\ast_u(|\mathbb{Z}|))$ by the previous exact sequence.
\end{proof}

We saw the homotopy stabilities as in Theorem \ref{thm_stability} and Proposition \ref{prp_stability_BSW}.
Then it is sufficient to investigate the inclusion $\U_\infty(C^\ast_u(|\mathbb{Z}|))\to\U_\infty(B^{\mathrm{SW}})$.

\begin{lem}
\label{lem_CuZ_BSW_even}
The inclusion $\U_\infty(C^\ast_u(|\mathbb{Z}|))\to\U_\infty(B^{\mathrm{SW}})$ induces isomorphisms on $\pi_{2i}$ for $i\ge0$.
\end{lem}

\begin{proof}
By Lemma \ref{lem_KBSW}, $K_1(B^{\mathrm{SW}})$ is isomorphic to $\mathbb{Z}$ and generated by the shift $S\in B^{\mathrm{SW}}$.
Since $S\in C^\ast_u(|\mathbb{Z}|)$ and $K_1(C^\ast_u(|\mathbb{Z}|))\cong\mathbb{Z}$, the map $K_1(C^\ast_u(|\mathbb{Z}|))\to K_1(B^{\mathrm{SW}})$ is an isomorphism.
Thus the map $\pi_{2i}(\U_\infty(C^\ast_u(|\mathbb{Z}|)))\to\pi_{2i}(\U_\infty(B^{\mathrm{SW}}))$ is also an isomorphism.
\end{proof}

Let $F_1$ be the homotopy fiber of the inclusion $\U_\infty(C^\ast_u(|\mathbb{Z}|))\to\U_\infty(B^{\mathrm{SW}})$.

\begin{prp}
\label{prp_htpytypeF}
The space $F_1$ has the homotopy type of the product of Eilenberg--MacLane spaces
\[
\prod_{i\ge1}^\circ
	K(\ell^\infty(\mathbb{Z},\mathbb{Z})_{S},2i-1).
\]
where $\ell^\infty(\mathbb{Z},\mathbb{Z})_{S}$ is a rational vector space of uncountable dimension.
\end{prp}

\begin{proof}
Observing the homotopy exact sequence
\[
\cdots
\to
\pi_i(F_1)
\to
\pi_i(\U_\infty(C^\ast_u(|\mathbb{Z}|)))
\to
\pi_i(\U_\infty(B^{\mathrm{SW}}))
\to
\pi_{i-1}(F_1)
\to
\cdots,
\]
we can see that the homotopy fiber inclusion $F_1\to\U_\infty(C^\ast_u(|\mathbb{Z}|))$ induces an isomorpshim on $\pi_{2i-1}$ and $\pi_{2i}(F_1)=0$ by Lemma \ref{lem_CuZ_BSW_even} and the fact that $\pi_{2i-1}(\U_1(B^{\mathrm{SW}}))=0$.
By Proposition \ref{prp_K(CuZ)}, we have $\pi_{2i-1}(F_1)\cong\ell^\infty(\mathbb{Z},\mathbb{Z})_S$.
The abelian group $\ell^\infty(\mathbb{Z},\mathbb{Z})_S$ is a rational vector space of uncountable dimension as seen in \cite[Section 5]{KKT}.
By \cite[Lemma 5.4]{KKT}, $F_1$ has the homotopy type of the product of Eilenberg--MacLane spaces as above.
\end{proof}

The following easy lemma is useful to study the homotopy type of the unitary group of a $C^\ast$-algebra.
Let $p_r\in M_d(\mathbb{C})$ denote the projection of rank $r$.

\begin{lem}
\label{lem_Uinfty_to_A}
Let $A$ be a $C^\ast$-algebra, where we do not require the existence of unit.
For any element $u\in K_0(A)$, there exists a (non-unital in general) $\ast$-homomorphism $f\colon\mathbb{C}\to M_d(A)$ such that $f_\ast[p_1]\in K_0(M_d(A))\cong K_0(A)$ equals to $u$.
\end{lem}

\begin{proof}
We can find a projection $p\in M_d(\mathbb{C}\oplus A)$ and $r\ge0$ such that $u=[p]-[p_r]$ in $K_0(A)$.
Define a $\ast$-homomorphism $f\colon\mathbb{C}\to M_d(A)$ by $f(1)=p$.
This is the desired map.
\end{proof}

\begin{prp}
\label{prp_section_CuZtoBSW}
The inclusion $\U_\infty(C^\ast_u(|\mathbb{Z}|))\to\U_\infty(B^{\mathrm{SW}})$ admits a homotopy section, which is an infinite loop map.
\end{prp}

\begin{proof}
Consider the inclusion of based loop spaces $\U_\infty(C_0(\mathbb{R},C^\ast_u(|\mathbb{Z}|)))\to\U_\infty(C_0(\mathbb{R},B^{\mathrm{SW}}))$.
By Proposition \ref{prp_infloop}, Lemma \ref{lem_Uinfty_to_A} and $K_0(C_0(\mathbb{R},C^\ast_u(|\mathbb{Z}|)))\cong\mathbb{Z}$, there exists an infinite loop map $f\colon\U_\infty(\mathbb{C})\to\U_\infty(C_0(\mathbb{R},C^\ast_u(|\mathbb{Z}|)))$ which induces an isomorphism on $\pi_{2i-1}$ for any $i\ge1$.
It follows from this and Lemma \ref{lem_CuZ_BSW_even} that the composite
\[
\U_\infty(\mathbb{C})
\xrightarrow{f}
\U_\infty(C_0(\mathbb{R},C^\ast_u(|\mathbb{Z}|)))
\to
\U_\infty(C_0(\mathbb{R},B^{\mathrm{SW}}))
\]
is a homotopy equivalence.
Then the inclusion of based loop spaces $\U_\infty(C_0(\mathbb{R},C^\ast_u(|\mathbb{Z}|)))\to\U_\infty(C_0(\mathbb{R},B^{\mathrm{SW}}))$ admits a homotopy section.
This implies that the inclusion of the double loop space $\U_\infty(C_0(\mathbb{R}^2,C^\ast_u(|\mathbb{Z}|)))\to\U_\infty(C_0(\mathbb{R}^2,B^{\mathrm{SW}}))$ also admits a homotopy section.
Thus the inclusion $\U_\infty(C^\ast_u(|\mathbb{Z}|))\to\U_\infty(B^{\mathrm{SW}})$ admits a homotopy section by Bott periodicity, which is again an infinite loop map.
\end{proof}

\begin{proof}[Proof of Theorem \ref{thm_decomp}]
By Proposition \ref{prp_section_CuZtoBSW}, we have a homotopy equivalence
\[
\U_\infty(C^\ast_u(|\mathbb{Z}|))\simeq\U_\infty(B^{\mathrm{SW}})\times F_1
\]
as infinite loop spaces.
The homotopy types of the spaces $\U_\infty(B^{\mathrm{SW}})$ and $F_1$ are determined in Lemma \ref{lem_pi0_BSW} and Proposition \ref{prp_htpytypeF}, respectively.
Together with the homotopy stability in Theorem \ref{thm_stability}, this completes the proof of the theorem.
\end{proof}

\section{Generalization}
\label{section_generalization}

In this section, we study the relation between the homotopy type of $\U_1(C^\ast_u(|\mathbb{Z}^n|))$ and the inclusion $\U_1(C^\ast_u(|\mathbb{Z}^n|))\subset\U_1(C^\ast(|\mathbb{Z}^n|))$ for general $n\ge2$.
In view of Theorem \ref{thm_decomp}, we propose the following question.
\begin{question}
Does the inclusion $\U_d(C^\ast_u(|\Gamma|))\to\U_d(C^\ast(|\Gamma|))$ admits a homotopy section?
Are the homotopy groups of its homotopy fiber are rational vector spaces?
\end{question}

Let us see the case when $\Gamma=\mathbb{Z}^n$ in view of this question.

\begin{lem}
\label{lem_KCZn}
The $K$-theory of the Roe algebra $C^\ast(|\mathbb{Z}^n|)$ is computed as
\[
K_i(C^\ast(|\mathbb{Z}^n|))
\cong
\begin{cases}
\mathbb{Z} & \text{$i\equiv n$ mod $2$,}\\
0 & \text{$i\not\equiv n$ mod $2$.}
\end{cases}
\]
\end{lem}

\begin{proof}
Let
\[
A_m=
\ell^\infty(\mathbb{Z}^n,\mathcal{K}(\mathcal{H}))\rtimes\mathbb{Z}^m
\]
with respect to the action of $\mathbb{Z}^m$ ($m\le n$) on the first $m$ factors of $\mathbb{Z}^n$.
Let $S_j$ denote the shift on the $j$-th factor.
Then by the Pimsner--Voiculescu exact sequence
\[
\xymatrix{
K_0(A_{m-1}) \ar[r]^-{1-S_m}
& K_0(A_{m-1}) \ar[r]
& K_0(A_m) \ar[d] \\
K_1(A_m) \ar[u]
& K_1(A_{m-1}) \ar[l]
& K_1(A_{m-1}) \ar[l]_-{1-S_m}
}
\]
for $A_m=A_{m-1}\rtimes_{S_m}\mathbb{Z}$, we obtain the short exact sequence
\begin{align}
\label{align_KAm}
0
\to
K_i(A_{m-1})_{S_m}
\to
&K_i(A_m)
\to
K_{1-i}(A_{m-1})^{S_m}
\to
0
\end{align}
for $i=0,1$, where $K_i(A_{m-1})_{S_m}$ and $K_i(A_{m-1})^{S_m}$ denote the coinvariant and the invariant by $S_m$, respectively.
Since $A_0=\ell^\infty(\mathbb{Z}^n,\mathcal{K}(\mathcal{H}))$ and we have the well-known isomorphism
\[
K_i(\ell^\infty(\mathbb{Z}^n,\mathcal{K}))
\cong
\begin{cases}
\mathbb{Z}^{\mathbb{Z}^n} & i=0,\\
0 & i=1,
\end{cases}
\]
where $\mathbb{Z}^{\mathbb{Z}^n}$ is the group of all $\mathbb{Z}$-valued sequences over $\mathbb{Z}^n$, we obtain
\[
K_i(A_m)
\cong
\begin{cases}
\mathbb{Z}^{\mathbb{Z}^{n-m}} & \text{$i\equiv m$ mod $2$,}\\
0 & \text{$i\not\equiv m$ mod $2$,}
\end{cases}
\]
by induction on $m$.
The lemma is just the case when $m=n$.
\end{proof}

Together with the previous lemma, the homotopy type of $\U_\infty(C^\ast(|\mathbb{Z}^n|))$ is determined by the following lemma.

\begin{lem}
\label{lem_stable_htpy_type_U_BU}
Let $A$ be a $C^\ast$-algebra, where we do not require the existence of unit.
Consider the following two conditions on $K$-theory:
\[
\text{(i)}\quad
K_i(A)\cong
\begin{cases}
\mathbb{Z} & i=0, \\
0 & i=1,
\end{cases}
\qquad\qquad
\text{(ii)}\quad
K_i(A)\cong
\begin{cases}
0 & i=0, \\
\mathbb{Z} & i=1.
\end{cases}
\]
If (i) holds, then $\U_\infty(A)$ has the homotopy type of $\U_\infty(\mathbb{C})$ as an infinite loop space.
If (ii) holds, then $\U_\infty(A)$ has the homotopy type of $\mathbb{Z}\times B\U_\infty(\mathbb{C})$ as an infinite loop space.
\end{lem}

\begin{proof}
Suppose the condition (i).
By Lemma \ref{lem_Uinfty_to_A}, there exists a homotopy equivalence $f\colon\U_\infty(\mathbb{C})\to\U_\infty(A)$, which is an infinite loop map.
When the condition (ii) holds, apply the result for the condition (i) to the algebra $C_0(\mathbb{R},A)$.
This implies that $\U_\infty(C_0(\mathbb{R},A))$ is homotopy equivalent to $\U_\infty(\mathbb{C})$.
By the Bott periodicity, $\U_\infty(A)$ is homotopy equivalent to $\Omega\U_\infty(\mathbb{C})\simeq\mathbb{Z}\times B\U_\infty(\mathbb{C})$.
\end{proof}

\begin{prp}
\label{prp_section_surj}
The inclusion $\U_\infty(C^\ast_u(|\mathbb{Z}^n|))\to\U_\infty(C^\ast(|\mathbb{Z}^n|))$ admits a homotopy section as an infinite loop map if and only if the homomorphism $K_\ast(C^\ast_u(|\mathbb{Z}^n|))\to K_\ast(C^\ast(|\mathbb{Z}^n|))$ is surjective.
\end{prp}
\begin{proof}
The only if part is obvious.
For the if part, when $n$ is odd, this follows from Lemma \ref{lem_KCZn} and the same argument as in the proof of Proposition \ref{prp_section_CuZtoBSW}.
When $n$ is even, apply the same argument to the map on the based loop spaces $\U_\infty(C_0(\mathbb{R},C^\ast_u(|\mathbb{Z}^n|)))\to\U_\infty(C_0(\mathbb{R},C^\ast(|\mathbb{Z}^n|)))$.
Then the proposition follows from the existence of the homotopy section of the map on the double loop spaces $\U_\infty(C_0(\mathbb{R}^2,C^\ast_u(|\mathbb{Z}^n|)))\to\U_\infty(C_0(\mathbb{R}^2,C^\ast(|\mathbb{Z}^n|)))$ and the Bott periodicity.
\end{proof}

Now all we have to do is to see that the homomorphism $K_\ast(C^\ast_u(|\mathbb{Z}^n|))\to K_\ast(C^\ast(|\mathbb{Z}^n|))$ is surjective.
Let
\[
B_m=
\ell^\infty(\mathbb{Z}^n,\mathbb{C})\rtimes\mathbb{Z}^m
\]
with respect to the action $\mathbb{Z}^m$ ($m\le n$) on the first $m$ factors of $\mathbb{Z}^n$ and $S_j$ denote the shift on the $j$-th factor.
We obtain the short exact sequences similar to (\ref{align_KAm})
\begin{align}
\label{align_KBm}
0
\to
K_i(B_{m-1})_{S_m}
\to
&K_i(B_m)
\to
K_{1-i}(B_{m-1})^{S_m}
\to
0
\end{align}
for $i=0,1$.
For $n=1,2$, we can see the surjectivity as follows.

\begin{lem}
\label{lem_surj_KCuZ}
The homomorphism $K_1(C^\ast_u(|\mathbb{Z}|))\to K_1(C^\ast(|\mathbb{Z}|))$ is an isomorphism.
\end{lem}
\begin{proof}
Consider the commutative square
\[
\xymatrix{
K_1(C^\ast_u(|\mathbb{Z}|)) \ar[r]^-\cong \ar[d]
& \ell^\infty(\mathbb{Z},\mathbb{Z})^S \ar[d]^-\cong \\
K_1(C^\ast(|\mathbb{Z}|)) \ar[r]^-\cong
& (\mathbb{Z}^\mathbb{Z})^S
}
\]
obtained from the exact sequences (\ref{align_KAm}) and (\ref{align_KBm}).
Thus the lemma follows.
\end{proof}

\begin{lem}
\label{lem_surj_KCuZ2}
The homomorphism $K_0(C^\ast_u(|\mathbb{Z}^2|))\to K_0(C^\ast(|\mathbb{Z}^2|))$ is surjective.
\end{lem}
\begin{proof}
When $n=2$, we can compute $K_\ast(B_1)$ by the exact sequence (\ref{align_KBm}) as follows:
\[
K_i(B_1)
\cong
\begin{cases}
\ell^\infty(\mathbb{Z}^2,\mathbb{Z})_{S_1} & i=0,\\
\ell^\infty(\mathbb{Z}^2,\mathbb{Z})^{S_1} & i=1.
\end{cases}
\]
Again by the exact sequences (\ref{align_KAm}) and (\ref{align_KBm}) for $m=2$, we have the commutative diagram
\[
\xymatrix{
0 \ar[r]
& \ell^\infty(\mathbb{Z}^2,\mathbb{Z})_{S_1S_2} \ar[r] \ar[d]
& K_0(C^\ast_u(|\mathbb{Z}^2|)) \ar[r] \ar[d]
& \ell^\infty(\mathbb{Z}^2,\mathbb{Z})^{S_1S_2} \ar[r] \ar[d]^-{\cong}
& 0 \\
0 \ar[r]
& 0 \ar[r]
& K_0(C^\ast(|\mathbb{Z}^2|)) \ar[r]^-{\cong}
& (\mathbb{Z}^{\mathbb{Z}^2})^{S_1S_2} \ar[r]
& 0
}
\]
Thus the homomorphism $K_0(C^\ast_u(|\mathbb{Z}^2|))\to K_0(C^\ast(|\mathbb{Z}^2|))$ is surjective by the right square.
\end{proof}

To determine the homotopy type of $C^\ast_u(|\mathbb{Z}^2|)$, we also need its $K$-theory.

\begin{lem}
\label{lem_KCuZ2}
The $K$-theory $K_1(C^\ast_u(|\mathbb{Z}^2|))$ and the kernel of the homomorphism $K_0(C^\ast_u(|\mathbb{Z}^2|))\to K_0(C^\ast(|\mathbb{Z}^2|))$ are rational vector spaces of uncountable dimension.
\end{lem}
\begin{proof}
As seen in the proof of Lemma \ref{lem_surj_KCuZ2}, the latter group is isomorphic to $\ell^\infty(\mathbb{Z}^2,\mathbb{Z})_{S_1S_2}$.
The coinvariant $\ell^\infty(\mathbb{Z}^2,\mathbb{Z})_{S_1}$ can be seen to be a rational vector space of uncountable dimension by the same argument as in \cite[Section 5]{KKT}.
Then, since $S_2\colon\ell^\infty(\mathbb{Z}^2,\mathbb{Z})_{S_1}\to\ell^\infty(\mathbb{Z}^2,\mathbb{Z})_{S_1}$ is a linear map on a rational vector space, the coinvariant $\ell^\infty(\mathbb{Z}^2,\mathbb{Z})_{S_1S_2}$ is a rational vector space of uncountable dimension.
For $K_1(C^\ast_u(|\mathbb{Z}^2|))$, we obtain the exact sequence
\[
0
\to
(\ell^\infty(\mathbb{Z}^2,\mathbb{Z})^{S_1})_{S_2}
\to
K_1(C^\ast_u(|\mathbb{Z}^2|))
\to
(\ell^\infty(\mathbb{Z}^2,\mathbb{Z})_{S_1})^{S_2}
\to
0
\]
from (\ref{align_KBm}).
Since $(\ell^\infty(\mathbb{Z}^2,\mathbb{Z})^{S_1})_{S_2}\cong\ell^\infty(\mathbb{Z},\mathbb{Z})_S$ and $\ell^\infty(\mathbb{Z}^2,\mathbb{Z})_{S_1}$ are rational vector spaces, $K_1(C^\ast_u(|\mathbb{Z}^2|))$ is also a rational vector space of uncountable dimension.
\end{proof}

\begin{proof}[Proof of Theorem \ref{thm_decomp_Z2}]
By Proposition \ref{prp_section_surj} and Lemma \ref{lem_surj_KCuZ2}, the inclusion $\U_\infty(C^\ast_u(|\mathbb{Z}^2|))\to\U_\infty(C^\ast(|\mathbb{Z}^2|))$ admits a homotopy section as an infinite loop map.
Let $F_2$ be the homotopy fiber of the inclusion.
Then we have a homotopy equivalence
\[
\U_\infty(C^\ast_u(|\mathbb{Z}^2|))
\simeq
\U_\infty(C^\ast(|\mathbb{Z}^2|))\times F_2
\]
as infinite loop spaces.
By Lemmas \ref{lem_KCZn} and \ref{lem_stable_htpy_type_U_BU}, $\U_\infty(C^\ast(|\mathbb{Z}^2|))$ is homotopy equivalent to $\U_\infty(\mathbb{C})$ as an infinite loop space.
By the naturality of the Bott maps
\[
\beta\colon\U_\infty(C^\ast_u(|\mathbb{Z}^2|))
\xrightarrow{\simeq}
\U_\infty(C_0(\mathbb{R}^2,C^\ast_u(|\mathbb{Z}^2|)))
\quad\text{and}\quad
\beta\colon\U_\infty(C^\ast(|\mathbb{Z}^2|))
\xrightarrow{\simeq}
\U_\infty(C_0(\mathbb{R}^2,C^\ast(|\mathbb{Z}^2|))),
\]
we have the homotopy equivalence
\[
\tilde{\beta}\colon F_2\xrightarrow{\simeq}\Omega^2 F_2
\]
as well.
The homotopy group of $F_2$ can be computed by Lemma \ref{lem_surj_KCuZ2}:
\[
\pi_i(F_2)\cong
\begin{cases}
V_1 & \text{$i$ is even,}\\
V_0 & \text{$i$ is odd,}
\end{cases}
\]
where
\[
V_0
=
\ker[K_0(C^\ast_u(|\mathbb{Z}^2|))\to K_0(C^\ast(|\mathbb{Z}^2|))],
\quad
V_1
=
K_1(C^\ast_u(|\mathbb{Z}^2|))
\]
are rational vector spaces by Lemma \ref{lem_KCuZ2}.
Again as in the proof of \cite[Lemma 5.4]{KKT}, we can find maps
\[
\prod_{i\ge1}^\circ K(V_0,2i-1)
\to
F_2
\quad\text{and}\quad
\prod_{i\ge1}^\circ K(V_1,2i-1)
\to
\Omega F_2
\]
inducing isomorphisms on the odd degree homotopy groups.
Taking the loop of the latter map, we get the composite 
\[
V_1\times\prod_{i\ge1}^\circ K(V_1,2i)
\to
\Omega^2F_2
\xrightarrow{\tilde\beta^{-1}}
F_2,
\]
which induces isomorphisms on the even degree homotopy groups and the group of path components.
Thus, we obtain the homotopy equivalence
\[
V_1\times\prod_{i\ge1}^\circ(K(V_0,2i-1)\times K(V_1,2i))
\to
F_2.
\]
This completes the proof of the theorem.
\end{proof}
Moreover, Lemma \ref{lem_surj_KCuZ} provides another proof of Theorem \ref{thm_decomp} in a similar manner.

\bibliographystyle{alpha}
\bibliography{uniformRoe}

\begin{thebibliography}{Roe03}

\bibitem[Bas64]{MR174604}
H.~Bass.
\newblock {$K$}-theory and stable algebra.
\newblock {\em Inst. Hautes \'{E}tudes Sci. Publ. Math.}, (22):5--60, 1964.

\bibitem[Bla86]{MR859867}
Bruce Blackadar.
\newblock {\em {$K$}-theory for operator algebras}, volume~5 of {\em
  Mathematical Sciences Research Institute Publications}.
\newblock Springer-Verlag, New York, 1986.

\bibitem[KKT]{KKT}
T.~Kato, D.~Kishimoto, and M.~Tsutaya.
\newblock Homotopy type of the space of finite propagation unitary operators on
  {$\mathbb{Z}$}.
\newblock {\em preprint, arxiv: 2007.06787}.

\bibitem[Kui65]{MR179792}
Nicolaas~H. Kuiper.
\newblock The homotopy type of the unitary group of {H}ilbert space.
\newblock {\em Topology}, 3:19--30, 1965.

\bibitem[MT21]{MR4155292}
Vladimir Manuilov and Evgenij Troitsky.
\newblock On {K}uiper type theorems for uniform {R}oe algebras.
\newblock {\em Linear Algebra Appl.}, 608:387--398, 2021.

\bibitem[PV80]{MR587369}
M.~Pimsner and D.~Voiculescu.
\newblock Exact sequences for {$K$}-groups and {E}xt-groups of certain
  cross-product {$C^{\ast} $}-algebras.
\newblock {\em J. Operator Theory}, 4(1):93--118, 1980.

\bibitem[Rie83]{MR693043}
Marc~A. Rieffel.
\newblock Dimension and stable rank in the {$K$}-theory of
  {$C^{\ast}$}-algebras.
\newblock {\em Proc. London Math. Soc. (3)}, 46(2):301--333, 1983.

\bibitem[Rie87]{MR887221}
Marc~A. Rieffel.
\newblock The homotopy groups of the unitary groups of noncommutative tori.
\newblock {\em J. Operator Theory}, 17(2):237--254, 1987.

\bibitem[Roe88]{MR918459}
John Roe.
\newblock An index theorem on open manifolds. {I}, {II}.
\newblock {\em J. Differential Geom.}, 27(1):87--113, 115--136, 1988.

\bibitem[Roe03]{MR2007488}
John Roe.
\newblock {\em Lectures on coarse geometry}, volume~31 of {\em University
  Lecture Series}.
\newblock American Mathematical Society, Providence, RI, 2003.

\bibitem[SW85]{MR783348}
Graeme Segal and George Wilson.
\newblock Loop groups and equations of {K}d{V} type.
\newblock {\em Inst. Hautes \'{E}tudes Sci. Publ. Math.}, (61):5--65, 1985.

\end{thebibliography}
\end{document}